%% file: preprint.tex
\documentclass[preprint,12pt]{article}

\usepackage[english]{babel}	% Deutsche Sprachanpassung, z.B. Silbentrennung bei Worten mit Sonderzeichen
\usepackage[utf8x]{inputenc}		% Direkte Eingabe von Umlauten
\usepackage{enumerate}
\usepackage{textcomp}                % für erweiterten Text-Symbolvorrat
\usepackage{typearea}
\usepackage{pdfpages}
\usepackage[toc]{appendix}
\usepackage{subcaption}
\usepackage{verbatim}
\usepackage{comment}
\usepackage{mathtools}
\usepackage[hidelinks]{hyperref}
\mathtoolsset{showonlyrefs,showmanualtags}

\usepackage{enumitem}
\usepackage{amsfonts, amsmath, amssymb, amsthm, dsfont}
\usepackage{amsthm}			% Eigene Umgebungen
\usepackage{thumbpdf}	% Seitenvorschau in PDF-Dokumenten
\usepackage{natbib}
\allowdisplaybreaks

\usepackage{booktabs}
\usepackage[noblocks]{authblk}
\usepackage{ulem}

\usepackage{pgf, tikz}
\usepackage{tikz-cd}		% Kommutative Diagramme
\usepackage{bbm}

\theoremstyle{definition}
\newtheorem{definition}{Definition}[section]

\newtheorem*{assumption*}{Assumption}

\newtheorem*{condition*}{Condition}
\newtheorem*{example}{Example}
\theoremstyle{plain}
\newtheorem{theorem}[definition]{Theorem}
\newtheorem{proposition}[definition]{Proposition}
\newtheorem{lemma}[definition]{Lemma}

\theoremstyle{remark}

\usepackage{graphicx}

\newcommand{\R}{\mathbb{R}}

\newcommand{\E}{\mathbb{E}}

\newcommand{\Var}{\operatorname{Var}}

\newcommand{\iid}{\overset{iid}{\sim}}

\excludecomment{proofcomplete}
\title{Empirical Orlicz norms} 
\author{Fabian Mies\thanks{Correspondence: f.mies@tudelft.nl\\ This publication is part of the project VI.Veni.242.365 of the NWO Talent Programme which is financed by the Dutch Research Council (NWO) under the grant https://doi.org/10.61686/AGAFX90293.}\\ Delft University of Technology}

\begin{document}

\maketitle

\begin{abstract}
\input{abstract.tex}


\end{abstract}

\input{content.tex}

\bibliography{literature.bib}
\bibliographystyle{apalike}
\end{document}

%% file: abstract.tex
The empirical Orlicz norm based on a random sample is defined as a natural estimator of the Orlicz norm of a univariate probability distribution. 
A law of large numbers is derived under minimal assumptions. 
The latter extends readily to a linear and a nonparametric regression model.
Secondly, sufficient conditions for a central limit theorem with a standard rate of convergence are supplied. 
The conditions for the CLT exclude certain canonical examples, such as the empirical sub-Gaussian norm of normally distributed random variables.
For the latter, we discover a nonstandard rate of $n^{1/4} \log(n)^{3/8}$, with a heavy-tailed, stable limit distribution.
It is shown that in general, the empirical Orlicz norm does not admit any uniform rate of convergence for the class of distributions with bounded Orlicz norm.

%% file: content.tex
\section{Motivation}

The Orlicz norm of a random variable $X$ taking values in a Banach space is defined as
\begin{align*}
    \|X\|_{\psi} = \inf\left\{ \sigma>0\,\Big|\, \E \psi\left(\tfrac{|X|}{\sigma}\right) \leq 1 \right\},
\end{align*}
where $\psi:[0,\infty)\to [0,\infty)$ is an Orlicz function, i.e.\ $\psi$ is increasing and convex such that $\psi(0)=0$ and $\psi(x)\to\infty$ as $x\to\infty$. 
These norms provide a convenient framework to formulate tail bounds, as Markov's inequality implies $P(|X|>t)\leq 1/\psi(t/\|X\|_\psi)$.
For this reason, Orlicz norms are widely employed in probability and data science, e.g.\ in high-dimensional probability theory (\cite{Vershynin2003}), empirical process theory and chaining arguments (\cite{van_der_vaart_weak_2023,van_de_geer_bernsteinorlicz_2013}), online learning (\cite{baby_online_2019}), shape-constrained inference (\cite{bellec_sharp_2018}), and multi-armed bandit problems (\cite{bubeck_regret_2012,hao_bootstrapping_2019}).
For the specific choice $\psi(x)=|x|^p$, we recover the usual $L_p$ norm. 
Another prominent special case are the exponential Orlicz norms with $\psi_\alpha(x)=\exp(|x|^\alpha)-1$ for $\alpha\geq 1$, corresponding to the class of sub-Weibull distributions \citep{vladimirova_subweibull_2020,zajkowski_norms_2020,kuchibhotla_moving_2022}.
In particular, the Orlicz function $\psi_1(x) = \exp(|x|)-1$ quantifies sub-exponential tails, and $\psi_2(x)=\exp(x^2)-1$ corresponds to sub-Gaussian distributions. 
Sub-Gaussianity has found direct use in statistics, e.g.\ as a quality criterion for robust mean estimators (\cite{Devroye2016,lugosi_sub-gaussian_2019,hopkins_mean_2020}, as a tightness condition to ensure weak convergence of certain multiscale test statistics \citep{kohne_at_2025}, and for threshold selection of  sequential monitoring procedures \citep{yu_note_2023}.

Although Orlicz-type tail bounds, and in particular sub-Gaussian bounds, are common assumptions for the asymptotic analysis of statistical methods, its empirical validation has surprisingly not been studied in the literature. 
Here, we consider estimation of the Orlicz norm $\|X\|_\psi$ based on a sample of $n$ iid random variables $X_1,\ldots, X_n\sim X$ via the natural estimator 
\begin{align*}
    \widehat{\sigma}_\psi(X_1,\ldots, X_n) = \inf\left\{ \sigma>0\,\Big|\, \frac{1}{n}\sum_{i=1}^n \psi\left(\tfrac{|X_i|}{\sigma}\right) \leq 1 \right\}.
\end{align*}
We refer to this estimator as the \textit{empirical Orlicz norm}.
Note that due to monotonicity, the empirical Orlicz norm can be computed efficiently via bisection.
We find that this estimator is generally consistent, while a rate of convergence is only available under additional assumptions. 
Even for standard Gaussian observations, the empirical sub-Gaussian norm is not asymptotically normal and exhibits a non-standard rate of convergence. 
While this demonstrates that estimation of Orlicz norms is in principle feasible, it also unveils interesting probabilistic phenomena.
The intention of this note is not to provide a complete statistical methodology, but to illustrate these somewhat unexpected behaviors of a natural estimator.

\section{Law of large numbers}

Our first result is a law of large numbers for the empirical Orlicz, which only requires that $\|X\|_{\psi}<\infty$.

\begin{theorem}[Law of large numbers]\label{thm:LLN}
    For $X_1,\ldots, X_n\iid X$ such that $\|X\|_{\psi}<\infty$,
    \begin{align*}
        \widehat{\sigma}_\psi(X_1,\ldots, X_n) \overset{n\to\infty}{\longrightarrow} \sigma_\psi=\|X\|_\psi \quad \text{almost surely.}
    \end{align*}
\end{theorem}

There are various statistical problems beyond the iid setting where control of Orlicz norms is necessary. 
For instance, the consistent variable selection of the LASSO method for linear regression requires the regularization parameter to be large enough compared to the sub-Gaussian norm of the errors \citep{wainwright_sharp_2009}.
The empirical Orlicz norm may be adapted to yield a consistent estimator in this model.
To this end, suppose that we observe 
\begin{align}
    Y_i=\beta^T Z_i + \epsilon_i, \qquad i=1,\ldots, n,
\end{align}
for a coefficient vector $\beta\in\R^d$, covariates $Z_i\in \R^d$, and iid centered error terms $\epsilon_i\in\R$. 
The Orlicz norm $\sigma_\psi=\|\epsilon\|_{\psi}$ of the noise can be estimated consistently by the residuals of a linear regression fit, i.e.\ as
\begin{align*}
    \widehat{\sigma}_{\psi, LM} 
    &= \widehat{\sigma}_{\psi}\left(Y_1-\widehat{\beta}^TZ_1,\ldots, Y_n-\widehat{\beta} Z_n\right),
\end{align*}
where $\widehat{\beta}$ is an estimator of the regression coefficients.

\begin{theorem}\label{thm:LM}
    Suppose that $Z_i$ are iid, and that $\widehat{\beta}\to \beta$ almost surely as $n\to\infty$.
    If $\sigma_\psi=\|\epsilon\|_{\psi}<\infty$ and $\|Z_{i,j}\|_\psi < \infty$ for each coordinate $j=1,\ldots, d$, then
    \begin{align*}
        \widehat{\sigma}_{\psi, \textsc{lm}} \to \sigma_\psi \qquad \text{almost surely as } n\to\infty.
    \end{align*}
\end{theorem}

The empirical Orlicz norm can also be applied to the nonparametric regression model 
\begin{align}
    Y_i = \mu_i +\epsilon_i,
\end{align}
for iid centered error terms $\epsilon_i$, and a nonparametric signal $\mu$. 
As usual in nonparametric regression, some regularity assumptions for the signal $i\mapsto \mu_i$ are necessary. 
We formulate these assumptions in terms of the exceedence numbers 
\begin{align*}
    \mathcal{E}_n(\mu,r)=\sum_{i=2}^n \mathds{1}(|\mu_i-\mu_{i-1}|>r), \qquad r>0.
\end{align*}
In the more common regression setup $\mu_i=f_i=f(i/n)$ for some function $f:[0,1]\to\R$, continuity of $f$ implies that $\mathcal{E}_n(f,r)\to 0$ as $n\to\infty$, for any $r>0$.
Moreover, if $f$ is right-continuous with left-hand limits, then 
\begin{align}
    \limsup_{n\to\infty} \mathcal{E}_{n}(\mu,r) \leq \mathcal{E}^*(r) \overset{r\to\infty}{\longrightarrow}0. \label{eqn:regularity}
\end{align}
See \cite[Lemma 1]{Billingsley1999} for the corresponding property of cadlag functions. 
This regularity assumption is much milder than requiring the function $f$ to be Hölder continuous, as common in nonparametric statistics, see e.g.\ \cite{Tsybakov2008}.

To estimate the Orlicz norm of the noise, we suggest the difference based estimator
\begin{align*}
    \widehat{\sigma}_{\psi, \textsc{np}} = \inf\left\{\sigma>0\,\Big|\, \frac{1}{n-1} \sum_{i=2}^n \psi\left( \frac{|Y_i-Y_{i-1}|}{\sigma} \right) \leq 1 \right\}.
\end{align*}

\begin{theorem}\label{thm:NP-diff}
    If $\|\epsilon\|_{\psi}<\infty$ and \eqref{eqn:regularity} holds, then
    \begin{align*}
        \widehat{\sigma}_{\psi,\textsc{np}} \to \|\epsilon_2-\epsilon_1\|_{\psi} \qquad \text{almost surely as }n\to\infty. 
    \end{align*}
\end{theorem}

The difference-based estimator $\frac{1}{n-1} \sum_{i=1}^n (Y_i-Y_{i-1})^2$, corresponding to the special case $\psi(x)=x^2)$, is commonly employed to estimate the second moment $\Var(\epsilon)$.
In the latter case, the variance of the individual $\epsilon_i$ may be recovered because $\Var(\epsilon_2\color{blue}-\color{black}\epsilon_1) = 2\Var(\epsilon)$.
However, there is no similar relation for general Orlicz norms, and thus the difference-based estimator can not simply be transformed to a consistent estimator of $\|\epsilon\|_\psi$.
Yet, the convexity of $\psi$ allows us to conclude $\E \psi((\epsilon_2-\epsilon_1)/\sigma)\leq \E \psi(\epsilon/\sigma)$ for any $\sigma>0$ by virtue of Jensen's inequality and $\E\epsilon=0$, which implies $\|\epsilon\|_{\psi}\leq \|\epsilon_2-\epsilon_1\|_{\psi}$.
Thus, the estimate $\widehat{\sigma}_{\psi,\textsc{np}}$ can be used as a conservative upper bound for the Orlicz norm of the error terms, which is sufficient for many statistical applications.

\section{Central limit theorem}

Under stronger moment assumptions on $X$, i.e.\ requiring more than just $\|X\|_\psi <\infty$, the empirical Orlicz norm also satisfies a central limit theorem. 

\begin{theorem}[Central limit theorem]\label{thm:CLT}
Suppose that the Orlicz function $\psi$ is continuously differentiable, such that (i) $\E\psi(|X|/\sigma_\psi)^2<\infty$ and (ii) $\E |X| \psi'(|X|/\sigma)<\infty$ for some $\sigma<\sigma_\psi$. 
Then, as $n\to\infty$
\begin{align*}
    \sqrt{n}\left(\widehat\sigma_\psi(X_1,\ldots, X_n) - \sigma_\psi\right)
    \overset{d}{\longrightarrow}\mathcal{N}\left(0,\frac{\E\psi\left(\frac{|X|}{\sigma_\psi}\right)^2-1}{\left[\E \frac{|X|}{\sigma_\psi^2} \psi'\left(\frac{|X|}{\sigma_\psi}\right)\right]^2 }\right).
\end{align*}
A sufficient condition for (i) and (ii) is that $\E\psi\left(\frac{|X|}{\sigma}+1\right)^2<\infty$ for some $\sigma<\sigma_\psi$.
\end{theorem}

For the exponential Orlicz norms $\|\cdot\|_{\psi_\alpha}$, the condition of Theorem \ref{thm:CLT} is satisfied if $\E \psi_\alpha(|X|/\sigma)<\infty$ for some $\sigma< 2^{-1/\alpha}\sigma_{\psi_\alpha}$. 
For example, any bounded random variable $X$ satisfies this condition for all $\alpha$. 
However, the conditions for the central limit theorem are more restrictive than those for the law of large numbers. 
Indeed, various non-standard rates of convergence can occur, even for elementary probability distributions. 
This is demonstrated via the following examples.

% Rates of convergence for the empirical Orlicz norm different from the standard $\sqrt{n}$ occur not only in pathological cases, but also for very elementary distributions, as the following examples demonstrate.

\begin{example}[Exponential distribution]
    For the standard exponential distribution $X\sim \text{Exp}(1)$, the sub-exponential norm is $\sigma_{\psi_1}=\|X\|_{\psi_1} =2$, but $\E \psi_1(X/\sigma_{\psi_1})^2=\infty$.
    Thus, the law of large numbers holds, but the rate of convergence is slower than $\sqrt{n}$. 
    Indeed, $P(\psi_1(X/2)>z) = (z+1)^{-2}$, and the generalized central limit theorem of Gnedenko and Kolmogorov, see e.g.\ \cite[Thm.~3.12]{nolan_univariate_2020}, yields $\frac{1}{\sqrt{n\log n}} \sum_{i=1}^n \left[\psi(X_i/2)-1\right] \overset{d}{\longrightarrow} \mathcal{N}(0, \tfrac{1}{2})$.
    Via formula \eqref{eqn:taylor} in the appendix, and using $\E [\frac{|X|}{\sigma_{\psi_1}^2} \exp(\frac{|X|}{\sigma_{\psi_1}})]=1$, we obtain the following result.
\end{example}

    \begin{proposition}\label{prop:Exp}
        Let $X_1,\ldots, X_n$ be iid standard exponentially distributed.
        Then
        \begin{align*}
        \sqrt{\tfrac{n}{\log n}}\left(\widehat{\sigma}_{\psi_1}(X_1,\ldots, X_n)-\sigma_{\psi_1}\right) \overset{d}{\longrightarrow} \mathcal{N}\left(0,\tfrac{1}{2}\right).
    \end{align*}
    \end{proposition}

\begin{example}[Weibull distribution]
    For the standard Weibull distribution $X\sim\mathrm{Wei}(1,\gamma)$ with shape parameter $\gamma>0$, i.e.\ $P(X>x)=\exp(-x^{\gamma})$, the conditions of Theorem \ref{thm:CLT} are satisfied for $\psi_{\alpha}$ for all $\alpha<\gamma$.
    For the boundary case $\alpha=\gamma$, a direct computation yields $\sigma_{\psi_\gamma}=\|X\|_{\psi_\gamma}=2^{1/\gamma}$.
    Moreover, for any $z>0$,
    \begin{align*}
        P\left( \psi_\gamma(X/\sigma_{\psi_\gamma}) >z\right) = P\left(X>\sigma_{\psi_\gamma} \log(z+1)^{1/\gamma}\right) = (z+1)^{-2},
    \end{align*}
    and
    \begin{align*}
        \E \left[\tfrac{|X|}{\sigma_{\psi_\gamma}^2} \psi_{\gamma}'\left(\tfrac{|X|}{\sigma_\psi}\right)\right] 
        &= \frac{\gamma}{2^{1/\gamma}}\E \left[\left(\tfrac{X}{\sigma_{\psi_\gamma}}\right)^\gamma \exp\left( \left(\tfrac{X}{\sigma_{\psi_\gamma}}\right)^\gamma\right)\right] \\
        &=\frac{\gamma^2}{2} \frac{1}{2^{1/\gamma}} \int_0^\infty z^{2\gamma-1} e^{-z^\gamma/2}\, dz 
        \quad = \frac{\gamma}{2} \frac{1}{2^{1/\gamma}} \int_0^\infty x e^{-x/2}\, dx = \frac{2\gamma}{2^{1/\gamma}}.
    \end{align*}
    Thus, we obtain analogously to the exponential case the asymptotic normality of the empirical Orlicz norm.
\end{example}

\begin{proposition}\label{thm:Weibull}
        Let $X_1,\ldots, X_n\sim \mathrm{Wei}(1,\gamma)$ be iid standard Weibull distributed.
        Then
        \begin{align*}
        \sqrt{\tfrac{n}{\log n}}\left(\widehat{\sigma}_{\psi_\gamma}(X_1,\ldots, X_n)-\sigma_{\psi_\gamma}\right) \overset{d}{\longrightarrow} \mathcal{N}\left(0,\tfrac{4^{1/\gamma}}{8\gamma^2}\right).
    \end{align*}
\end{proposition}

\begin{example}[Normal distribution]
    If $X\sim\mathcal{N}(0,1)$, then the exponential-type Orlicz function $\psi_\alpha$ satisfies Theorem \ref{thm:CLT} for all $\alpha\in[1,2)$ because $\E\exp(a|X|^\alpha)<\infty$ for all $a>0$.
    However, for the sub-Gaussian case, $\alpha=2$, a direct computation yields $\sigma_{\psi_2}=\|X\|_{\psi_2}=\sqrt{8/3}$, but Theorem \ref{thm:CLT} does not apply because $\E \exp\left(X^2/(8/3)\right)^2=\infty$.
    Instead, we derive the tail bound
    \begin{align*}
        P\left(\psi_2\left(\tfrac{|X|}{\sigma_{\psi_2}}\right)-1>z\right) 
        &= 2\,P\left(X>\sqrt{\tfrac{8}{3}\log(z+2)} \right) \\
        &= \exp\left( - \frac{\tfrac{8}{3}\log(z+2)}{2} \right) L(z) 
        = z^{-\frac{4}{3}} \widetilde{L}(z),
    \end{align*}
    for slowly varying functions $L, \widetilde{L}$, i.e.\ $\tilde{L}(az)/\tilde{L}(z)\to 1$ as $z\to \infty$ for any $a>0$, see Lemma \ref{lem:gauss-quantiles}. 
    Thus, the sum $\sum_{i=1}^n \psi_2(X_i/\sigma_{\psi_2})$ admits a $\beta$-stable limit distribution with index $\beta=\frac{4}{3}$, which leads to a heavy-tailed limit of the empirical sub-Gaussian norm.
\end{example}

    \begin{proposition}\label{prop:Gauss}
        Let $X_1,\ldots, X_n$ be iid standard Gaussian. 
        Then
        \begin{align*}
            n^{\frac{1}{4}} \log(n)^{\frac{3}{8}} \left[\widehat{\sigma}_{\psi_2}(X_1,\ldots, X_n)-\sigma_{\psi_2}\right] \overset{d}{\longrightarrow} \sqrt{\tfrac{2}{27 \pi^{3/4}}}(Y-4),        \end{align*}
        for a fully right-skewed $\beta$-stable random variable $Y$ with stability index $\beta=\frac{4}{3}$, and characteristic function given by \eqref{eqn:Y-charfun} in the appendix.
    \end{proposition}

Beyond these examples, without any further regularity assumptions, it is not possible to derive a rate of  convergence for the empirical Orlicz norm.

\begin{theorem}[No rate of convergence]\label{prop:LB}
    Consider the exponential-type Orlicz function $\psi(x)=\psi_\alpha(x)=\exp(|x|^\alpha)-1$ for some $\alpha\in[1,2]$.
    For any $\beta>0$, there exists a random variable $X$ with $\|X\|_\psi=1$ such that for any sequence $X_1,\ldots, X_n\iid X$, 
    \begin{align*}
        \limsup_{n\to\infty} \frac{|\widehat{\sigma}_\psi(X_1,\ldots, X_n) -\sigma_\psi| }{n^{-\beta}} = \infty \qquad \text{almost surely}.
    \end{align*}
\end{theorem}

This highlights that the empirical Orlicz norm does not admit any parametric rate of convergence uniformly on the class $\mathcal{X}_\psi=\{X:\|X\|_\psi\leq 1\}$ of all distributions with bounded Orlicz norm. 
Note that Theorem \ref{prop:LB} does not rule out a potential uniform logarithmic rate of convergence. 
However, we conjecture that there does not exist any uniform rate of convergence for the empirical Orlicz norm, as the following statistical lower bound holds for any estimator.

\begin{theorem}\label{thm:uniform-lb}
    For any Orlicz function $\psi$ and any estimator $\widetilde{\sigma}_\psi=\widetilde{\sigma}_\psi(X_1,\ldots,X_n)$ with $X_1,\ldots, X_n \iid X$, and any rate $r_n\to 0$, 
    \begin{align*}
        \limsup_{n\to\infty}\sup_{X\in\mathcal{X}_\psi} P\left(  \widetilde{\sigma}_\psi(X_1,\ldots, X_n) - \|X\|_\psi > r_n\right) = 1.
    \end{align*}
\end{theorem}

The claim of Theorem \ref{thm:uniform-lb} is stronger in that it also excludes a slower-than-polynomial rate for any estimator, but requires a stricter form of uniformity than Theorem \ref{prop:LB}.
Specifically, in Theorem \ref{prop:LB}, the distribution of $X$ is fixed for all $n$, while Theorem \ref{thm:uniform-lb} allows for $X$ to vary with $n$.
We also note that in any parametric family of distributions, the Orlicz norm can be estimated via plug-in of a parameter estimator, which typically converges at rate $\sqrt{n}$. The empirical Orlicz norm is, on the other hand, a model-free estimator.

\begin{example}
    Data-driven estimates of the probability $\overline{F}(t)=P(X>t)$, for very large $t$ are typically derived via extreme value theory, and applied, for example, in hydrology \citep{katz_statistics_2002}.
    This mathematical approach is able to yield accurate probability estimates for any $t\leq t_n$ in the regime $\overline{F}(t_n)\gg \exp(-\sqrt{n})$, see \citet[Thm.~4.4.1]{de_haan_extreme_2006}.
    As an alternative to the extreme value methodology, the empirical Orlicz norm can be plugged into the conservative tail inequality $\overline{F}(t)\leq\overline{F}^*(t)= 1/\psi(t/\|X\|_\psi)$ to obtain an empirical upper bound $\overline{F}^\star(t)=1/\psi(t/\widehat{\sigma}_\psi)$, which upholds for very large values of $t$. 
    Specifically, rainfall measurements are often described as having sub-Weibull tails \citep{papalexiou_battle_2013,marra_non-asymptotic_2023}, which naturally leads to the use of sub-Weibull norms $\sigma_{\psi_\alpha}=\|X\|_{\psi_\alpha}$.
    The tail bound takes the form
    \begin{align*}
        P(X>t) &\leq\overline{F}^*(t)= \frac{1}{\exp(t^\alpha/\sigma_{\psi_\alpha}^\alpha)-1} 
        \leq \frac{1}{\exp(t^\alpha/\widehat{\sigma}_{\psi_\alpha}^\alpha)-1} \Delta_n(t) = \overline{F}^\star(t)\Delta_n(t), 
    \end{align*}
    where, for some $\overline{\sigma}$ between $\sigma_{\psi_\alpha}$ and $\widehat{\sigma}_{\psi_\alpha}$,
    \begin{align*}
        \left|\Delta_n(t)-1\right| 
        &\leq \left|\frac{\exp(t^\alpha/\widehat{\sigma}_{\psi_\alpha}^\alpha) - \exp(t^\alpha/{\sigma}_{\psi_\alpha}^\alpha)}{\exp(t^\alpha/\sigma_{\psi_\alpha}^\alpha)-1} \right| 
        \quad \leq \frac{\alpha t^\alpha \left|\widehat{\sigma}_{\psi_\alpha}-\sigma_{\psi_\alpha}\right|}{\overline{\sigma}^{\alpha+1}} \frac{\exp(t^\alpha/\overline{\sigma}^\alpha)}{\exp(t^\alpha/\sigma_{\psi_\alpha}^\alpha) -1 }   \\
        &\leq \frac{\alpha t^\alpha |\widehat{\sigma}_{\psi_\alpha}-\sigma_{\psi_\alpha}|}{\overline{\sigma}^{\alpha+1}} \frac{\exp(t^\alpha/\sigma_{\psi_\alpha}^\alpha)}{\exp(t^\alpha/\sigma_{\psi_\alpha}^\alpha) -1 }  \exp\left( \tfrac{\alpha t^\alpha}{\overline{\sigma}^{\alpha+1}} |\widehat{\sigma}_{\psi_\alpha}-\sigma_{\psi_\alpha}|\right)
    \end{align*}
    We thus find that
    $\Delta_n(t)\to 1$ uniformly for $t\in[0,t_n]$ if $t_n \ll |\widehat{\sigma}_{\psi_\alpha} -  {\sigma}_{\psi_\alpha}|^{1/\alpha}$. 
    Hence, the rate of convergence of $\widehat{\sigma}_\psi$ governs how far we can reliably extrapolate into the tail. 
    Under the conditions of Theorem \ref{thm:CLT}, the rate implies the bound $t_n \ll n^{2/\alpha}$, corresponding to a tail probability bound of approximately $\overline{F}^\star(t_n)=\exp(-o(n^2))$. In comparison to the extreme value approach, we obtain an upper bound instead of an exact estimate, but the former is reliable way farther into the tail.
\end{example}

\section*{Proofs}
\subsection*{Proof: Law of large numbers (Theorem \ref{thm:LLN})}

Suppose that $P(X\neq 0)>0$ without loss of generality. 
Then $G:[0,\infty)\to[0,\infty],\sigma\mapsto \E \psi(|X|/\sigma)$ is strictly decreasing, and $G(\sigma)\to \infty$ as $\sigma\to\infty$.
Thus for any $\delta>0$, we find $G(\sigma_\psi+\delta)<1$ and $G(\sigma_\psi-\delta)>1$. 
Denoting $G_n(\sigma) = \frac{1}{n}\sum_{i=1}^n \psi(|X_i|/\sigma)$, the law of large numbers yields that almost surely
\begin{align*}
    G_n(\sigma_\psi-\delta)\to G(\sigma_\psi-\delta)>1, \qquad G_n(\sigma_\psi+\delta)\to G(\sigma_\psi+\delta)<1.
\end{align*}
and thus, almost surely, \begin{align*}
    \sigma_\psi-\delta \leq \liminf_{n\to\infty} \widehat{\sigma}_\psi \leq \limsup_{n\to\infty} \widehat{\sigma}_\psi \leq \sigma_\psi+\delta.
\end{align*}
As $\delta>0$ is arbitrary, this establishes the convergence of $\widehat{\sigma}_\psi$.

\subsection*{Proof: Regression models (Theorem \ref{thm:LM} and Theorem \ref{thm:NP-diff})}

\begin{proof}[Proof of Theorem \ref{thm:LM}]
    The triangle inequality for Orlicz norms gives for any $\overline{\beta}\in\R^d$
    \begin{align*}
        \widehat{\sigma}_{\psi}(Y_i-\bar{\beta}^T Z_i)  
        &= \widehat{\sigma}_{\psi}\left(\epsilon_i+(\beta-\bar{\beta})^T Z_i\right) \\
        &= \widehat{\sigma}_{\psi}(\epsilon_i) + \Delta, \\
        \text{for some }\qquad |\Delta|
        &\leq \sum_{j=1}^d \widehat{\sigma}_{\psi}\left((\beta - \bar\beta)_j Z_{i,j}\right) \\
        &\leq \sum_{j=1}^d |\beta_j-\bar\beta_j| \widehat{\sigma}_\psi(Z_{1,j},\ldots, Z_{n,j}).
    \end{align*}
    The law of large numbers for empirical Orlicz norms yields $\widehat{\sigma}_\psi(Z_{1,j},\ldots, Z_{n,j})\to \|Z_{1,j}\|_\psi$ almost surely for each $j$, and $\widehat{\sigma}_{\psi}(\epsilon_1,\ldots, \epsilon_n)\to \|\epsilon\|_{\psi}=\sigma_\psi$.
    Thus, if $\widehat{\beta}\to \beta$, then $\widehat{\sigma}_{\psi,\textsc{lm}}$ is consistent.
\end{proof}

\begin{proof}[Proof of Theorem \ref{thm:NP-diff}]
    By the triangle inequality for Orlicz norms, we find
    \begin{align*}
        \widehat{\sigma}_{\psi,\textsc{np}} 
        &= \widehat{\sigma}_\psi(\epsilon_2-\epsilon_1,\ldots, \epsilon_n-\epsilon_{n-1}) + \Delta_n, \\
        |\Delta_n| &\leq \widehat{\sigma}_{\psi}(\mu_2-\mu_1,\ldots, \mu_n-\mu_{n-1}).
    \end{align*}
    By the law of large numbers, $\widehat{\sigma}_\psi(\epsilon_2-\epsilon_1,\ldots, \epsilon_n-\epsilon_{n-1})\to \|\epsilon_2-\epsilon_1\|_\psi$ almost surely.

    To bound $\Delta_n$, let $R>0$ such that $\mathcal{E}^*(R)=0$, and choose an arbitrary $r>0$.
    Moreover, choose $N$ large enough such that $\mathcal{E}_n(\mu,R)=0$ and $\mathcal{E}_n(\mu,r)\leq 2\mathcal{E}^*(r)$ for all $n\geq N$. 
    Then, for $n\geq N$, and $q=\psi^{-1}(\frac{1}{2})$,
    \begin{align*}
        \frac{1}{n-1} \sum_{i=2}^n \psi\left( \frac{|\mu_i-\mu_{i-1}|}{r/q} \right)
        &\leq \frac{2\mathcal{E}^*(r)}{n-1} \psi\left(\frac{Rq}{r}\right) + \psi\left(\frac{r}{r/q}\right) 
        \overset{n\to\infty}{\longrightarrow} \psi(q)=\frac{1}{2}.
    \end{align*} 
    Hence, $\lim\sup \widehat{\sigma}_\psi(\mu_2-\mu_1,\ldots, \mu_n-\mu_{n-1}) \leq r/q$. 
    Since $r>0$ is arbitrary, we conclude that $\Delta_n\to 0$ as $n\to\infty$.
\end{proof}

\subsection*{Proof: Central limit theorem (Theorem \ref{thm:CLT})}
 
Convexity and monotonicity of $\psi$ imply that there exists a $x_0$ such that $\psi(x)\geq cx$ for all $x>x_0$. 
Thus, for $|x|>x_0$, and for some $C$, 
\begin{align*}
    0\leq -\frac{d}{d\sigma} \psi\left(\frac{|x|}{\sigma}\right) = \frac{|x|}{\sigma^2} \psi'\left(\frac{|x|}{\sigma}\right) 
    &\leq \frac{C}{\sigma} \psi\left(\frac{|x|}{\sigma}\right) \psi'\left(\frac{|x|}{\sigma}\right) \\
    &\leq \frac{C}{\sigma} \psi\left(\frac{|x|}{\sigma}\right) \int_{|x|/\sigma}^{|x|/\sigma +1}\psi'\left(z\right)\, dz
    \quad \leq \frac{C}{\sigma} \psi\left( \frac{|x|}{\sigma}+1 \right)^2.
\end{align*}
Moreover, $\psi(|x|/\sigma_\psi)^2\leq \psi(|x|/\sigma +1)^2$, which establishes the sufficient condition claimed in the Theorem.
We also note that 
\begin{align}
    G'(\sigma)\leq 0, \quad \text{and $G'$ is increasing.}\label{eqn:G}
\end{align}

To proof the central limit theorem, let $G_n$ and $G$ as above. 
We use the condition $\E |X| \psi'(|X|/\sigma^*) <\infty$ for some $\sigma^*<\sigma_\psi$ to conclude $G_n'(\sigma)\to G'(\sigma)$ almost surely, for all $\sigma \geq \sigma^*$.
Since convexity of $\psi$ implies that $G_n'$ is monotonically increasing, the latter convergence is uniform on $[\sigma^*,\infty)$.
Moreover, a Taylor expansion yields
\begin{align*}
    1  = G_n(\widehat{\sigma}_\psi) 
    = G_n(\sigma_\psi) + G_n(\widehat{\sigma}_\psi) - G_n(\sigma_\psi) 
    = G_n(\sigma_\psi) + [\widehat{\sigma}_\psi-\sigma_\psi] G_n'(\widetilde{\sigma}_\psi),
\end{align*}
for some $\widetilde{\sigma}_\psi$ between $\sigma_\psi$ and $\widehat{\sigma}_\psi$.
The strong consistency of $\widehat{\sigma}_\psi$ implies $\widetilde{\sigma}_\psi\to \sigma_\psi$ almost surely.
Moreover, $G'(\sigma_\psi)<0$, because otherwise \eqref{eqn:G} yields $G'(z)=0$ for all $z\geq \sigma_\psi$ which would lead to the contradiction $1=G(\sigma_\psi)=\lim_{z\to \infty}G(z)=0$.
Combining the fact that $G'(\sigma_{\psi})\neq 0$ with the locally uniform consistency of $G_n'$, 
\begin{align}
    \widehat{\sigma}_\psi - \sigma_\psi = - \frac{G_n(\sigma_\psi)-1}{G'(\widetilde{\sigma}_\psi) + o_P(1)} 
    = - \frac{G_n(\sigma_\psi)-\E G_n(\sigma_\psi)}{G'(\widetilde{\sigma}_\psi) + o_P(1)}. \label{eqn:taylor}
\end{align}
Note that continuity of $\psi$ and $G(\sigma^*)<\infty$ imply the continuity of $G$, and thus $1=G(\sigma_\psi) = \E(G_n(\sigma_\psi))$. 
Due to the condition $\E \psi(|X|/\sigma_\psi)^2<\infty$, the central limit theorem implies 
\begin{align*}
    \sqrt{n}[G_n(\sigma_\psi)-\E G_n(\sigma_\psi)] \overset{d}{\longrightarrow} \mathcal{N}\left(0, \Var \psi\left(\tfrac{|X|}{\sigma_\psi}\right)\right).
\end{align*} 
Via Slutsky's Lemma, we obtain the weak convergence of $\widehat{\sigma}_\psi$ since $G'(\sigma_\psi) = -\E \left[\frac{|X|}{\sigma_\psi^2}\psi'\left(\frac{|X|}{\sigma_\psi}\right)\right]$.

\subsection*{Proof: Gaussian example (Proposition \ref{prop:Gauss})}

As $\psi_2'(|x|)=2|x| \exp(|x|^2)$, we readily check that 
\begin{align}
    -G'(\sigma_{\psi_2}) 
    =\E \left[ \tfrac{|X|}{\sigma_{\psi_2}^2} \psi_2'\left(\tfrac{|X|}{\sigma_{\psi_2}}\right) \right] 
    &= \E\left[ \frac{2|X|^2}{\sigma_{\psi_2}^3} \exp\left(\tfrac{|X|^2}{\sigma_{\psi_2}^2}\right) \right]\\
    &= \sqrt{\frac{3}{8}}\frac{6}{8} \int_{-\infty}^\infty x^2 e^{\frac{3}{8} x^2} \frac{e^{-\frac{x^2}{2}}}{\sqrt{2\pi}}\,  dx \nonumber \\
    & = \sqrt{\frac{3}{8}}\frac{3}{2}\int_{-\infty}^\infty x^2 \frac{e^{-\frac{x^2}{8}}}{2\sqrt{2\pi }}\, dx \quad = \sqrt{\tfrac{27}{2}}. \label{eqn:Gauss-deriv}
\end{align}
Similarly, for any $\sigma>\sqrt{2}$ we find $\E|X|\psi_2'(|X|/\sigma)<\infty$.
As this expectation is finite, the representation \eqref{eqn:taylor} holds and we are led to study the asymptotics of 
\begin{align*}
    G_n(\sigma_{\psi_2}) - 1 = \frac{1}{n} \sum_{i=1}^n \left[\exp\left(\tfrac{X_i^2}{8/3}\right)-2\right].
\end{align*}
Denote $Z_i=\exp\left(\tfrac{X_i^2}{8/3}\right)-2$.
Since $P(Z_i>z)$ is regularly varying as $z\to \infty$, and zero as $z\to -\infty$, the generalized central limit theorem \cite[Thm.~3.7.2]{durrett_probability_2010} yields
    \begin{align*}
         \left( \sum_{i=1}^n Z_i  \;-b_n \right) \Big/ a_n \overset{d}{\longrightarrow} Y,       
    \end{align*}
    for a non-degenerate distribution $Y$, and 
    \begin{align*}
        a_n &= \left\{\inf x: P\left(\left|\psi_2\left(\tfrac{|X|}{\sigma_{\psi_2}}\right)-1\right|>x\right)\leq n^{-1}\right\} = \psi_2\left( \tfrac{1}{\sigma_{\psi_2}} \Phi^{-1}\left(1-\tfrac{1}{2n}\right) \right) -1, \\
            %%%%
        b_n&=n \E \left[ Z \cdot \mathbf{1}(|Z|\leq a_n) \right] = -n \E \left[ Z \cdot \mathbf{1}(|Z|> a_n) \right].
    \end{align*}
    The second equality for $b_n$ uses $\E(Z)=0$.
    The asymptotic expansion of $\Phi^{-1}$ presented in Lemma \ref{lem:gauss-quantiles} below
    yields
    \begin{align*}
        a_n 
        &= \exp\left[ \tfrac{3}{8} \left(\log (4n^2) - \log\log (2n) + \log\tfrac{1}{4\pi} + o(1)\right) \right]-2 \\
        &= \pi^{-3/8} n^{\frac{3}{4}} \log(2n)^{-3/8} (1+o(1)) \\
        &= \left(\pi \log n\right)^{-\frac{3}{8}} n^{\frac{3}{4}} (1+o(1)).
    \end{align*}
    We also note that $P(Z>a_n)=n^{-1}$ for any $n$ such that $a_n>1$, because $Z$ has a continuous distribution and $Z>-1$.
    
    To determine $b_n$, we define $c_n=\sigma_{\psi_2}(\psi_2-1)^{-1}(a_n)=\Phi^{-1}(1-\frac{1}{2n})$, and compute
    \begin{align*}
        \E \left[ Z \cdot \mathbf{1}(|Z|>a_n)\right] 
        &= \E\left[(Z-a_n)\mathds{1}(Z>a_n)\right] \,+\, a_n P(Z>a_n)\\
        &= \int_{a_n}^\infty P(Z>t)\, dt\,+\, \frac{a_n}{n} \quad = a_n\int_{1}^\infty P(Z>a_nt)\, dt \,+\, \frac{a_n}{n} \\
        &= a_n \int_1^\infty (a_n t)^{-\frac{4}{3}} \frac{\tilde{L}(a_nt)}{\tilde{L}(a_n)} \tilde{L}(a_n)\, dt \,+\, \frac{a_n}{n} \\
        &= a_n^{-\frac{1}{3}} \tilde{L}(a_n) \int_1^\infty t^{-\frac{4}{3}} \frac{\tilde{L}(a_nt)}{\tilde{L}(a_n)}\, dt \,+\, \frac{a_n}{n} \\
        &= a_n^{-\frac{1}{3}} \tilde{L}(a_n) \left[ \int_1^\infty t^{-\frac{4}{3}}dt + o(1) \right] \,+\, \frac{a_n}{n}\\
        &= a_n^{-\frac{1}{3}} \tilde{L}(a_n) \left[ 3 + o(1) \right] \,+\, \frac{a_n}{n},
    \end{align*}
    where the last step holds by dominated convergence and Karamata's representation theorem \cite[B.1.6]{de_haan_extreme_2006}, which states that $\tilde{L}(z)=c(z)\exp(\int_1^z a(s)/s\, ds)$ for functions $c(z)\to c_0$ and $a(z)\to 0$ as $z\to\infty$.
    Moreover,
    \begin{align*}
        a_n^{-\frac{1}{3}} \tilde{L}(a_n) = a_n P\left(\psi_2\left(\tfrac{|X|}{\sigma_{\psi_2}}\right)-1>a_n\right) = \frac{a_n}{ n},
    \end{align*}
    and thus $b_n=-4 a_n(1+o(1))$.
    Hence, we find that
    \begin{align*}
        \pi^{3/8} n^{1/4} \log(n)^{3/8}\left[\frac{1}{n} \sum_{i=1}^n Z_i\right] + 4 \;=\; 
         \frac{\sum_{i=1}^n Z_i +4a_n}{a_n(1+o(1))} 
         \;\overset{d}{\longrightarrow} \; Y. 
    \end{align*}
    As the tail of $Z_i$ is regularly varying with exponent $-\frac{4}{3}$, the limit $Y$ is $\beta$-stable for $\beta=\frac{4}{3}$ and fully right-skewed, see the proof of \cite[Thm.~3.7.2]{durrett_probability_2010}.
    In particular, its characteristic function is
    \begin{align}
        \E e^{\mathbf{i}\lambda Y} = \exp\left( \tfrac{4}{3} \int_0^\infty \left(e^{\mathbf{i}\lambda x} - 1 - \mathbf{i}\lambda x \cdot \mathds{1}_{x<1}\right)x^{\frac{4}{3}-1}\, dx \right). \label{eqn:Y-charfun}
    \end{align}
    In view of the asymptotic expansion \eqref{eqn:taylor} and the explicit value \eqref{eqn:Gauss-deriv} for the denominator, we have established the convergence $n^{1/4} \log(n)^{3/8}[\widehat{\sigma}_{\psi_2}-\sigma_{\psi_2}]\to (Y-4)/\pi^{3/8}/\sqrt{27/2}$. \qed

The following asymptotic appears to be mathematical folklore, but we were unable to find a proper reference. 
Hence, we include its derivation for completeness.

\begin{lemma}[Gaussian quantiles]\label{lem:gauss-quantiles}
    The standard Gaussian quantile function $\Phi^{-1}$ admits the asymptotic expansion
    \begin{align*}
        \Phi^{-1}(1-q) = \sqrt{\log\tfrac{1}{q^2} - \log\log \tfrac{1}{q} + \log\tfrac{1}{4\pi}  + o(1)}, \qquad \text{as } q\downarrow 0.
    \end{align*}
\end{lemma}
% \begin{remark}
%     Wikipedia states the asymptotic expansion $\Phi^{-1}(1-q) = \sqrt{\log\tfrac{1}{q^2} - \log\log \tfrac{1}{q^2} - \log(2\pi)}+o(1)$. The leading term is identical to our formula, but we provide a sharper error bound.
% \end{remark}
\begin{proof}[Proof of Lemma \ref{lem:gauss-quantiles}]
    Define $R(z)=2\sqrt{\pi}[1-\Phi(\sqrt{2z})]$.
    Two-sided tail bounds for the Gaussian distribution, see \cite[Lem.~VII.2]{feller_introduction_1968}, yield
    \begin{align*}
        \left( \tfrac{1}{x}-\tfrac{1}{x^3} \right)\;\leq\; 1-\Phi(x) \;\leq\; \tfrac{1}{x} \varphi(x), \qquad x\geq 0,
    \end{align*}
    for the standard Gaussian density $\varphi$, and thus
    \begin{align*}
         \left(1-\tfrac{1}{2z}\right)\frac{1}{\sqrt{z}}e^{-z}\leq R(z) &\leq  \frac{1}{\sqrt{z}}e^{-z}, \qquad z>0.
    \end{align*}
    For any small $\epsilon>0$, and $q \in (0,1/\sqrt{2})$, define $z_{\pm}(q) = \log \frac{1}{q} - \frac{1}{2}\log \log \frac{1}{q} \pm \epsilon$.
    Then
    \begin{align*}
        R\left(z_-(q)\right) 
        \leq \frac{q \sqrt{\log (1/q)} }{\sqrt{\log (1/q) - \frac{1}{2}\log\log (1/q) -\epsilon}} e^{-\epsilon} = q (1+o(1)) e^{-\epsilon},
    \end{align*}
    which is strictly smaller than $q$ for $q<q_-(\epsilon)$, and some $q_-(\epsilon)>0$. 
    Thus, $R^{-1}(q)>z_-(q)$.
    Similarly,
    \begin{align*}
        R\left(z_+(q)\right) 
        \geq  \left(1-\tfrac{1}{2z_+(q)}\right)\frac{q \sqrt{\log (1/q)} }{\sqrt{\log (1/q) - \frac{1}{2}\log\log (1/q) -\epsilon}} e^{+\epsilon} 
        = q(1+o(1)) e^{\epsilon},
    \end{align*}
    which is strictly larger than $q$ for $q<q_+(\epsilon)$, and some $q_+(\epsilon)>0$. 
    Thus, $R^{-1}(q)<z_+(q)$.
    Since $\epsilon>0$ is arbitrary, we conclude that
    \begin{align*}
        \lim_{q\to 0} \left\{ R^{-1}(q) - \left[  \log \frac{1}{q} - \frac{1}{2}\log \log \frac{1}{q}\right]\right\}
        = 0.
    \end{align*}
    To conclude the proof, observe that 
    \begin{align*}
        \Phi^{-1}(1-q)=\sqrt{2\cdot R^{-1}\left(2\sqrt{\pi} q\right)}
        &= \sqrt{\log\tfrac{1}{q^2} - \log\left(\log \tfrac{1}{q} + \log\tfrac{1}{2\sqrt{\pi}}\right) + \log\tfrac{1}{4\pi}  + o(1)} \\
        &=\sqrt{\log\tfrac{1}{q^2} - \log\log \tfrac{1}{q} + \log\tfrac{1}{4\pi}  + o(1)}.
    \end{align*}
    
\end{proof}

\subsection*{Proof: No rate of convergence}

\begin{proof}[Proof of Theorem \ref{prop:LB}]
    Let $\gamma\in(1,2)$, and $X$ be a random variable such that
    \begin{align*}
        P(\psi(X)>t) = P(X>\psi^{-1}(t))&=\min\left[1, \left(\frac{t\gamma}{\gamma-1}\right)^{-\gamma}\right] \\
        \iff \qquad P(X>z) &= \min\left[1, \left(\frac{\psi(z)\gamma}{\gamma-1}\right)^{-\gamma}\right].
    \end{align*}
    Since $\psi$ is increasing and tends to infinity, this is a valid survival function, and $X$ is supported on $[\psi^{-1}(\frac{\gamma-1}{\gamma}),\infty)$. 
    We verify that $\|X\|_\psi=1$ as
    \begin{align*}
        \E(\psi(|X|)) 
        &= \int_0^{\infty} P(\psi(X)>t)\, dt \\
        &= \int_0^\infty \min\left[1, \left(\frac{t\gamma}{\gamma-1}\right)^{-\gamma}\right]\, dz
        \quad =1.
    \end{align*}
    An analogous calculation shows that $\E(\psi(|X|)^{\gamma^*})<\infty$ for any $\gamma^*\in (1,\gamma)$.
    For the exponential-type Orlicz functions $\psi=\psi_\alpha$, we readily check that $x\psi_\alpha'(x) \leq C \psi_\alpha(x/r) \leq \psi_\alpha(x)^{\gamma^*}$ for $r<1$ sufficiently close to $1$. 
    Hence, $G_n'(r)\to G'(r)<\infty$ almost surely, and by monotonicity of $\psi$ this convergence is locally uniform. 
    We may hence conclude as in the proof of Theorem \ref{thm:CLT} that
    \begin{align*}
        \widehat{\sigma}_{\psi_\alpha} - \sigma_{\psi_\alpha} = \frac{G_n(1)-\E(G_n(1))}{G'(1)+o_P(1)}.
    \end{align*}
    However, in contrast to Theorem \ref{thm:CLT}, the addends $\psi_\alpha(X_i)$ comprising $G_n(1)$ have regularly varying upper tails with exponent $\gamma$, and finite lower tails.
    Thus, the generalized central limit theorem \cite[IX.8]{feller1971introduction} yields convergence of $n^{1-\frac{1}{\gamma}} [G_n(1)-\E G_n(1)]$ towards a fully right-skewed $\gamma$-stable random variable. 
    In particular, the estimation error $\widehat{\sigma}_{\psi_\alpha}-\sigma_{\psi_\alpha}$ is of exact order $n^{\frac{1}{\gamma}-1}$. 
    Since $\gamma\in(1,2)$ is arbitrary, any rate $n^\beta$ with $\beta\in (0,\frac{1}{2})$ can be obtained.                        
\end{proof}

\begin{proof}[Proof of Theorem \ref{thm:uniform-lb}]
    Let $X_1,\ldots, X_n\sim X = 0$ almost surely, and $Y_1^n,\ldots, Y_n^n \sim Y^n$ where $P(Y^n=\psi^{-1}(n^2))=1/n^2$ and $P(Y^n=0) = 1-1/n^2$. 
    Then $\|X\|_\psi = 0$ and $\|Y^n\|_{\psi}=1$, but any estimator satisfies 
    \begin{align*}
        P\left(\widetilde{\sigma}_\psi(X_1,\ldots, X_n) = \widetilde{\sigma}_\psi(Y_1^n,\ldots, Y_n^n)\right) 
        \;\geq\; P\left((X_1,\ldots, X_n) = (Y_1^n,\ldots, Y_n^n)\right) \geq 1-\tfrac{1}{n}.
    \end{align*}
    Hence, the estimator can not uniformly distinguish the cases $\sigma_\psi=0$ and $\sigma_\psi=1$.
\end{proof}